\documentclass{article}

\usepackage[utf8]{inputenc}
\usepackage{times}
\usepackage{a4wide,fullpage}

\usepackage[pdftex]{hyperref}
\usepackage{amsmath,amsthm,amsopn,amssymb,amsfonts,latexsym}
\usepackage{epsfig,graphics}
\usepackage{graphicx}
\usepackage{enumerate}
\usepackage{pgf,tikz,pgfplots}

\usepackage{array,delarray,rotating,mathtools}
\usepackage[mathscr]{eucal}
\pgfplotsset{compat=1.11}

\usepackage{xcolor,hyperref}
\hypersetup{
	colorlinks = true,
	allbordercolors = {white},
	urlcolor = {blue},
	citecolor = {black}
}

\theoremstyle{plain}
\newtheorem{theorem}{Theorem}[section]

\theoremstyle{definition}

\theoremstyle{remark}

\title{Practical Online Assessment of Mathematical Proof}
\author{Bickerton, R. and Sangwin, C.J.}
\date{April 2020}

\begin{document}
	
	\maketitle

\begin{abstract}
We discuss a practical method for assessing mathematical proof online. 
We examine the use of faded worked examples and reading comprehension questions to understand proof. 
By breaking down a given proof, we formulate a checklist that can be used to generate comprehension questions which can be assessed automatically online. 
We then provide some preliminary results of deploying such questions.
\end{abstract}

	
	\section{Introduction}\label{sec:int}
	
	Mathematical proof is central to the discipline of mathematics, indeed it is a hallmark which differentiates mathematics from other subjects.
	Mathematical proof is also very difficult to learn.
	At the start of a university course, where proof becomes more complex and more important, students are commonly taught through an explicit {\em introduction to proof} course.
	A typical mode of teaching is to use lectures to introduce particular types of proof, and then have students solve problems and prove conjectures of a similar type.
	Ultimately, students will be expected to select the type of proof, and correctly write a complete proof of a conjecture which will be assessed.

	Assessment is a very expensive use of staff time and it is increasingly common for universities to use online assessment systems, particularly to support methods-based parts of the syllabus \cite{2013CAA}.
	Online assessment of complete arguments, including proof, is much more difficult than assessment of a final answer.
	That said, it is not clear that students are best-served by the traditional mode of teaching, particularly when regular and detailed feedback on their written work is not available.
	Instead, we are investigating alternatives.
	In particular, we consider the following practical approaches to developing online assessments designed to support formative assessments in proof-based courses.
	\begin{enumerate}
		\item Faded worked examples.
		\item Explicit assessment of {\em separated concerns}.
		\item Reading comprehension.
	\end{enumerate}
	The last topic is the largest, with a number of separate aspects.
	These include (i) preparation for a proof, (ii) reading comprehension of a particular proof, and (iii) proof followup.
	This, we believe, will better serve students than the existing common practice of demonstrating a whole proof and then expecting students to imitate something very similar.
	
	Prior research of traditional assessment has generated useful insight into students' understanding of proof.
	\cite{Davies2019} suggests students' problems include content knowledge \cite{Moore1994}; overall strategy \cite{Weber2001} and an over-reliance on inappropriate argument forms \cite{Harel2007}.
	Reading a proof, {\sl pre-se}, is not an active process.
	Learning needs carefully structured, conscious activity with effort.
	Proof comprehension tasks provide such activity and are likely to help students focus their attention.
	\cite{Hodds2014} found that students who were trained to provide self-explanations performed better on a proof comprehension test than those who were not.
	However, \cite{Hodds2014} suggested that making key ideas visible, e.g. with explicit layout, is not as important as having students engage with a proof to make sure important aspects in the proof become clear to them.
	\cite{Hodds2014} has shown that the self-explanation strategy substantially improves students' comprehension of mathematical proofs.
	Asking students for warrants, missing steps, appears to require very similar activities on the part of the student.
	
	Previous research that compared the performance of experts and non-experts has revealed some counter-intuitive results.
	For example, the {\em expertise reversal effect} refers to the finding that instruction techniques can have opposite levels of effectiveness with expert and non-expert learners \cite{Kalyuga2003,Kalyuga2012}.
	Related is the {\em worked-example effect}, which refers to the observation that novices benefit more from studying worked-examples than from independent problem-solving, \cite{Renkl2004}.
	Both of these findings have implications for how we choose to introduce students to the complex task of writing mathematical proofs.

	This is a practical paper, and so we have to consider the tools we have available to us to assess students' answers online.
	Our ultimate goal is to help students write a complete proof of a conjecture.
	We cannot currently automatically assess students' free-form proof, indeed this looks like a particularly difficult task which will remain unsolved in the foreseeable future.
	That said, we give a strong preference, where possible, to asking students to provide an answer of their own rather than using multiple choice questions (MCQ) or similar question types.
	Some issues associated with the differences between multiple choice and constructed response were discussed detail in \cite{2017MCQ}.
	
	What tools do we have available to us?
	We assume we have an online assessment system which is capable of accepting algebraic expressions from a student as an answer and establishing objective mathematical properties of those answers.  Such a system is commonplace \cite{2013CAA}.
	We also assume we are able to automatically assess algebraic derivations, such as {\em reasoning by equivalence}.
	This technology is less common than accepting a final answer, but is increasingly being used in calculus and algebra courses.
	A specific example is given in Figure \ref{fig:induction2} in which the algebra of the induction step is entered free-form by the student and automatically assessed.
	For more details of this technology see \cite{2018Sangwin-equivalence-proof}.
	
	Multiple choice questions will certainly have their place in online assessment associated with proof, but they are very difficult to write and normally require testing and refining.
	We are aware of software which will assess students' short answers, typically consisting of single sentences \cite{Butcher2010}.
	We have not, yet, used the ``pattern match'' question type developed by \cite{Butcher2010}, but the kinds of questions we would like to ask might be automatically assessed with this technology.
	Note that the pattern match question type needs a data set, and requires expertise from staff to implement patterns to be used.
	So, these questions may prove to be as ``expensive'' to develop as MCQ.

	\section{Faded worked examples}
	
	According to  \cite{Renkl2004} a ``worked example'' consists of three components: a problem formulation, the solution steps, and the final solution itself.
	{\em Classic faded worked examples} refers to a progressive sequence of worked examples in which steps within a worked example are systematically removed, requiring students to take increasing responsibility for completing the problem.
	The use of faded worked examples is a form of scaffolding, and there are many choices of what can be faded.
	For example, removing steps from the end of the problem, i.e.~first removing the last step, has been found to be most favourable for learning \cite{kinnear_george_2019_2565969}.
	For supporting cognitive skill acquisition in well-structured domains,  \cite{Renkl2004} found that it is useful to use classic faded worked examples before starting to solve problems independently.
	This is particularly useful where there is a ``model worked solution'' which a student is expected to learn.
	
	This technique can be applied when teaching particular types of proof, e.g. when teaching proof by induction.
	However, this form of scaffolding intentionally removes much of the decision making from students.
	That is to say, the decision to provide a template removes the responsibility from the student to decide to use a proof by induction.
	This will reduce the cognitive load on a student, which might prove very helpful in an early formative situation.
	
	Figure \ref{fig:induction} shows an induction question in the STACK online assessment system, in which there is a lot of scaffolding.
	The left hand figure shows the blank question initially presented to the student.
	Note that one of the input boxes contains a ``syntax hint'' to reduce the difficulty of turning sigma notation such as \(\sum_{k=1}^n k^2\) into a linear typed expression \verb$sum(k^2,k,1,n)$.
	The right hand figure shows a completed student response, without any feedback about correctness.
	Typed expressions are displayed to the student in traditional notation, before they submit any response.
	
	\begin{figure}
		\begin{center}
			\includegraphics[width=8cm]{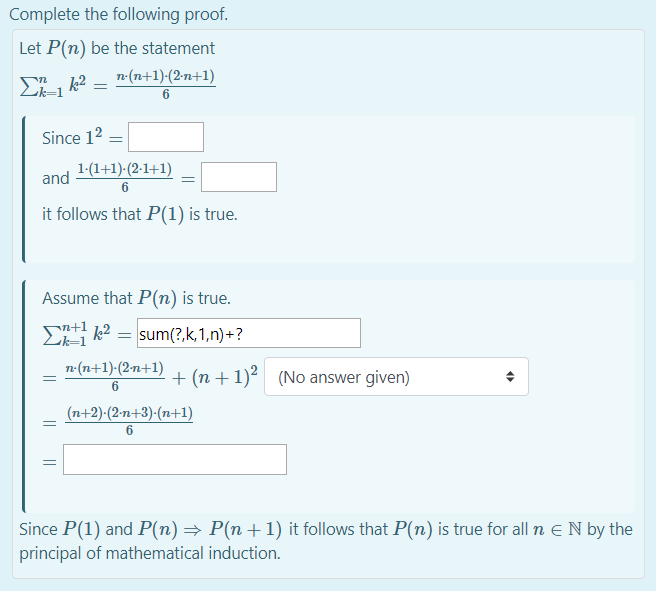}
			$~$
			\includegraphics[width=8cm]{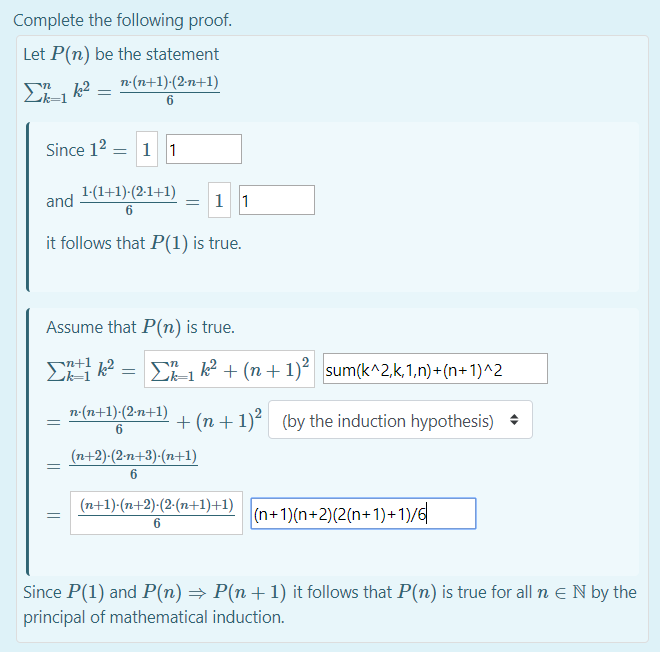}
			\caption{An induction problem with scaffolding, blank (left) and with a student's answers (right)}\label{fig:induction}
		\end{center}
	\end{figure}
	
	Here, almost all decisions have been made for the student, and the structure of the proof is essentially complete.
	In this case the student is only responsible for some of the algebraic steps, and for justifying one step with a multiple choice option.
	It is, within current online assessment technology, possible to assess line-by-line algebraic working.
	For this type of proof by induction, a student needs to know that they should be aiming to algebraically manipulate the formula to make it clear we have the right hand side of \(P(n+1)\) when written as a function of \(n+1\).
	In Figure \ref{fig:induction2}, the student can take any algebraic steps they please, and the software will establish (i) algebraic equivalence of adjacent steps, and (ii) that the last step is written in the correct form needed to show \(P(n+1)\) really does hold.
	While the ultimate goal of most university proof courses would be to have the student write the whole proof, we believe there is real merit in practice in more structured situations for novice students.
	
	\begin{figure}
		\begin{center}
			\includegraphics[width=9cm]{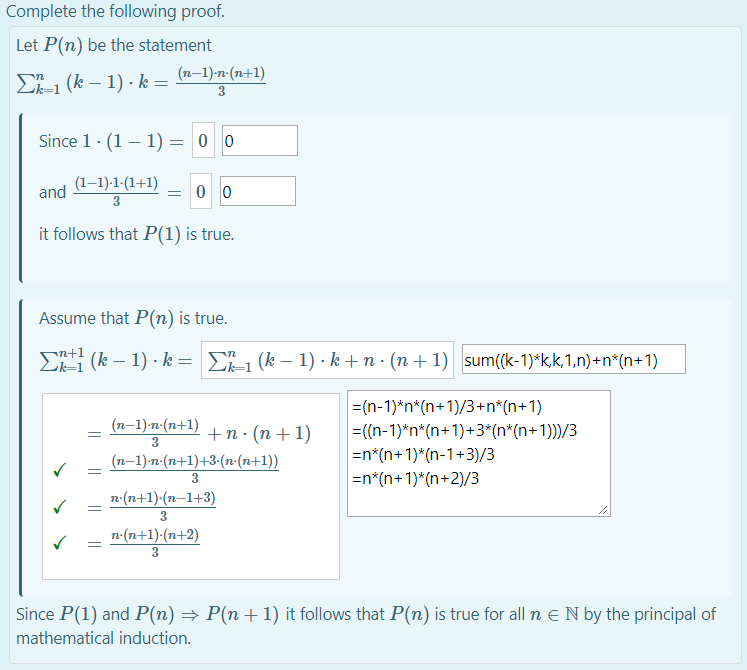}
			\caption{An induction problem in which students are responsible for the algebraic decision making}\label{fig:induction2}
		\end{center}
	\end{figure}
	
	\section{Separating concerns}
	
	A mathematical proof is typically a rather complex mix of formal statements, definitions, logical relationships and calculations.
	In a faded worked example, the worked solution is provided which students are expected to complete this.
	An alternative is to provide explicit practice of specific concerns which will very shortly occur in a proof.
	Much of this practice can be done online, even when students will be writing proofs on paper in a traditional way.
	
	For example, novice students typically struggle with \(\sum\)-notation in fundamentally mathematical ways.
	More specifically, students confuse and confound \(\sum_{k=1}^n n\) with \(\sum_{k=1}^n k\), or they do not notice the differences between
	\[
	\sum_{k=1}^n a_k,\
	\sum_{k=0}^n a_k,\
	\sum_{k=1}^{n+1} a_k \mbox{ and }
	\sum_{k=0}^{n+1} a_k,
	\]
	where the action takes place in the limits of the summation.
	A further problem is a general confusion of the status of local and global variables, and a lack of confidence over whether \( \sum_{k=1}^n a_k= \sum_{m=1}^n a_m = \sum_{m=0}^{n-1} a_{m+1}\).
	
	Other examples of specific concerns which can be separated include the following.
	\begin{itemize}
		\item Negation of logical statements, including quantified statements, e.g. those with \(\forall\) or \(\exists\).
		\item Re-writing expressions to make $n+1$ the variable in proof by induction, including just writing \(P(n+1)\).
		\item Algebraic manipulation of expressions with inequalities and absolute values in preparation for an \(\epsilon/\delta\)-arguments in analysis.
		``Secret working'', e.g. reverse engineering the \(\epsilon/\delta\)-argument.
	\end{itemize}
	An alternative to using extensive scaffolding in examples, such as those shown in Figure \ref{fig:induction}, is to explicitly {\em separate concerns}.
	Separating concerns refers to explicitly identifying, teaching and assessing specific topics in relative isolation in anticipation of their immanent need.
	Indeed, \cite{Osterholm2006} found that mathematics did not appear to be the most dominant factor affecting reading comprehension.
	Instead, the use of symbols was more relevant, suggesting explicit attention to symbolism should be given particularly when any new symbolism is introduced.
	
	An obvious danger with separated concerns is the lack of obvious motivation for a topic, or particular type of calculation.
	Indeed, one difficulty is knowing which concerns to separate in preparation for a particular topic.

	\section{Terminology: proof-gadgets and steps}\label{sec:term}

	Having names for things is important, especially when talking about them.
	For example, the word ``ansatz'' is widely used as ``an educated guess or an additional assumption made to help solve a problem, and which is later verified to be part of the solution by its results'' (Wikipedia).
	Many simple proofs by contradiction could be a contrapositive instead, see \cite{2018-Contrapositive} for a discussion.
	Having a word ``contrapositive'' is very helpful when discussing such proofs, and explaining to students the difference between contradicting the hypothesis, and a general external contradiction such as \(1=0\).
	
	Mathematical theorems can be divided into two classes: specific and general.
	Specific theorems concern one object, or a unique situation.
	For example, the following classic proofs are all specific: (i) there are infinitely many prime numbers, (ii) the real numbers are uncountable, and (iii) \(\sqrt{2}\not\in\mathbb{Q}\).
	General theorems have hypotheses which a range of examples do/do not satisfy, e.g.
	\begin{theorem}
		\label{th:bounded-inc-converge}
		If \( (a_n) \) is a bounded and increasing sequence then \( \lim_{n\rightarrow\infty} a_n\) exists.
	\end{theorem}
	This is a general theorem, with three definitions in play, namely (i) bounded, (ii) increasing, and (iii) convergent.
	There is more scope in general theorems for students to work with examples.
	Can students correctly identify which of the three properties the example satisfies?
	More challenging, can they provide their own examples?
	In the context of Theorem \ref{th:bounded-inc-converge} we can have a convergent sequence which is not increasing, but cannot have a convergent sequence which is not bounded.  Do students understand this significant difference between the two hypotheses?
	
	When proving a theorem it is common to construct some kind of object and then prove this has certain properties. E.g.~in the standard proof that there are infinitely many primes we assume the primes are finite \(p_1,\cdots, p_n\) and then consider \(n=p_1\times \cdots \times p_n +1\).  We then establish properties of \(n\) which lead to a contradiction.
	There does not appear to be a commonly used general name for such objects.
	In this paper will use the word {\em gadget}, or {\em proof-gadget} for emphasis, for a particular object constructed as a device within a proof, built to establish certain conditions must hold.
	The word gadget is often used to refer to a device, e.g. mechanical or electrical, with ingenious, novel and practical aspects.
	
	Similarly, some methods and techniques make use of what \cite{Pointon02} termed {\em facilitator objects}.  For example, in an \(\epsilon/\delta\)-argument the goal is to show that \(\forall\epsilon >0\exists \delta: \cdots\).
	Essentially, the proof must start ``Let \(\epsilon>0\) and take \(\delta = f(\epsilon)\), then ... ''.
	We consider defining the facilitator object \(f\) in \(\delta = f(\epsilon)\) to be a particular form of proof-gadget, worthy of the separate name.
	While the proof is typically presented in a finished form, to write the proof the function \(f\) has to be found by reverse-engineering, using {\em hidden working}.
	For proofs and techniques which make use of hidden working the teacher could choose to separate this concern as an explicit exercise, prior to the formal proof being written.  
	
	It is common to talk about {\em steps in the proof} in a somewhat loose sense.
	A mathematical proof is written as an ordered sequence of statements.
	English sentences contains one or more statements, and statements can be highly abbreviated by containing mathematical symbols.
	For example, the symbol $=$ was consciously introduced as a synonym {\em ``to avoid the tedious repetition of the words `is equals to'.''} \cite{Recorde1557}.
	Nearly all modern notation abbreviates, and nearly all notation can be read as part of a sentence.
	This does not help us understand what a mathematical ``step'' might be.
	
	A commonly used model for general arguments was devised by Toulmin \cite{Toulmin2003} and used in mathematics education \cite{Inglis2007}.
	Toulmin's scheme has six components.
	Data (D) is evidence on which the claim is based.
	The conclusion (C) is the actual claim which is being put forward.
	The warrant (W) gives the justification for deriving the conclusion from the data.
	The backing (B) are ``other assurances, without which the warrants themselves would posses[sic] neither authority nor currency'' \cite[p.~96]{Toulmin2003}.
	A qualifier (Q) gives the degree of confidence in the claim.
	Lastly, a rebuttal (R) gives circumstances in which the claim might not hold.
	These items are arranged in what Toulmin calls the {\em ``argumentation pattern'' } shown in Figure \ref{fig:Toulmin}.
	
	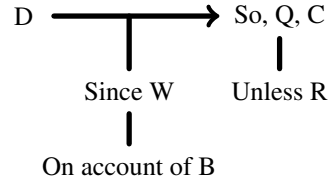
\begin{figure}
		\centering
		\begin{tikzpicture}[line cap=round,line join=round,x=2.0cm,y=2.0cm]
		\draw (0.3,1) node {D};
		\draw [->,line width=2.pt] (0.5,1.) -- (1.6,1.);
		\draw (2,1) node {So, Q, C};
		\draw (1,0.5) node  {Since W};
		\draw (1,0) node {On account of B};
		\draw (2,0.5) node {Unless R};
		\draw [line width=2.pt] (1.,1.)-- (1,0.65);
		\draw [line width=2.pt] (1,0.35)-- (1,0.15);
		\draw [line width=2.pt] (2,0.85)-- (2,0.65);
		\end{tikzpicture}
		\caption{Toulmin's argumentation pattern}
		\label{fig:Toulmin}
	\end{figure}
	
	Toulmin's scheme is not applied to a whole proof, and we do not propose to apply the scheme only to adjacent statements within a proof.
	Rather we acknowledge that a typical mathematical proof is a nested recursive structure.
	Toulmin's scheme will be applied both to adjacent statements and to ``blocks'' which make up the internal structure of the proof.
	Consider the following theorem, and its proof.
	
	\begin{theorem}
		If \(a+b\sqrt{2}=c+d\sqrt{2}\) and \(a,b,c,d \in \mathbb{Q}\) then \(a=c\) and \(b=d\).
	\end{theorem}
	
	\begin{proof}
		Suppose (for a contradiction) that \(b\neq d\). If \(a+b\sqrt{2}=c+d\sqrt{2}\) then, rearranging, we have
		\((a-c)=(d-b)\sqrt{2}\).
		Dividing gives \( \sqrt{2}=\frac{a-c}{d-b} \in \mathbb{Q}\).
		But [as previously proved] \(\sqrt{2}\not\in\mathbb{Q}\).
		This is a contradiction, so \(b=d\).
		Then setting \(b=d\) in \(a+b\sqrt{2}=c+d\sqrt{2}\) it follows \(a=c\).
	\end{proof}
	
	The following proof has a more structured presentation.
	
	\begin{proof}
		Assume  \(a+b\sqrt{2}=c+d\sqrt{2}\) and \(a,b,c,d \in \mathbb{Q}\).
		Then
		\begin{align*}
			a+b\sqrt{2}=c+d\sqrt{2} \\
			\Leftrightarrow (a-c)=(d-b)\sqrt{2}.
		\end{align*}
		\begin{enumerate}
			\item If \(b\neq d\) then \( \sqrt{2}=\frac{a-c}{d-b}\).  Note that since \(a,b,c,d \in \mathbb{Q}\) it follows \(\frac{a-c}{d-b}\in\mathbb{Q} \).  But [as previously proved] \(\sqrt{2}\not\in\mathbb{Q}\).
			This contradicts the assumption \(b\neq d\).
			\item If \(b=d\) then \((a-c)=0\), i.e.~\(a=c\), and the theorem holds.
		\end{enumerate}
		The only case which holds is \(b=d\) and so \(a=c\).
	\end{proof}
	Notice the second proof has the following nested structure.
    \begin{quote}
	    Equivalence reasoning.\\
	    Cases:
	    \begin{itemize}
	        \item \(b\neq d\): Contradiction.
	        \item \(b= d\): Direct proof.
	    \end{itemize}
     \end{quote}
	The proof contains some direct reasoning by equivalence before exhaustive cases on the equality of \(b\) and \(d\).
	The case \(b\neq d\) contains a contradiction, so cannot occur, leaving only the case \(b=d\).
	However, the contradiction is to the hypothesis of the sub-proof \(b\neq d\).
	An alternative proof by contradiction of the overall conclusion needs the hypothesis ``\(a\neq c\) or \(b\neq d\)'' and still requires exhaustive cases.
	
	
	There is still considerable philosophical discussion about the nature of legitimate mathematical arguments, and how to analyse proof.
	We will use the phrases ``data'', ``conclusion'' and ``warrant'' as in the Toulmin scheme.
	Warrants will apply to adjacent statements, e.g.~explicitly using a re-write rule such \(a=b \Leftrightarrow a+x=b+x\) when reasoning by equivalence.
	Warrants will apply more generally to statements within a proof, e.g. using properties of a definition.
	Backing will refer to more general properties external to a proof, e.g. the legitimacy of a type of proof, such as the equivalence of a statement and its contrapositive, or to statements it would be inappropriate to justify in detail in this context, such as that every integer is either odd or even.
	Note that the qualifier and rebuttal are often omitted when discussing formal mathematical arguments, see \cite{Inglis2007}.

	\section{Reading comprehension}
	
	Traditional assessment requires students to write proofs of their own, and it appears to us that it is only comparatively rarely that students are asked reading comprehension questions.
	In response \cite{Mejia-Ramos2012} identified the following seven aspects of proof comprehension.
	The first three are {\em local} and the last four consider {\em global} concerns.
	\begin{enumerate}
		\item {\em Logical status of statements and proof framework.}
		The phrase {\em proof framework} was also used by \cite{Selden1995}.
		\item  {\em Meaning of terms and statements within the proof.}
		This specifically includes understanding, and being able to use, mathematical definitions. \cite{Moore1994}.
		\item {\em Justification of claims.}
		\item {\em Summarizing via high-level ideas.}
		\item  {\em Identifying the modular structure.} Mathematical proof often has a recursive structure, with local arguments within a larger global argument.
		\item  {\em Transferring the general ideas or methods to another context.}
		\item {\em Illustrating with examples.}
	\end{enumerate}
	However, these all refer to reading comprehension of a particular proof.
	In practice, we suspect that a slightly broader view is needed.
	Indeed, this is acknowledged in the headings {\em Transferring the general ideas or methods to another context} which looks outside the current proof.
	
	\cite{Weber2015b} investigated reading strategies that students can use to productively engage with a proof.  Derived from interviews with students he reports the following strategies.
	\begin{enumerate}
		\item[1a] Understanding the theorem by rephrasing it in one’s own words
		\item[1b] Understanding the theorem by expressing it in logical notation
		\item[2] Trying to prove a theorem before reading its proof
		\item[3] Considering the proof framework used in the proof
		\item[4] Partitioning the proof into parts or sub-proofs
		\item[5] Checking confusing inferences with examples
		\item[6] Comparing the method used in the proof with one’s own approach
	\end{enumerate}
	\cite{Weber2015b} claims that a contribution of his paper is {\em ``the suggestion that students' understanding of the proofs that they read can be improved if they can be taught to apply the strategies''} he reports.
	Weber draws an interesting comparison with attempts, such as those of \cite{Polya54}, to teach problem-solving heuristics.
	
	\subsection{Proof understanding baseline checklist}
	
	Based on the seven aspects of proof comprehension identified by \cite{Mejia-Ramos2012}, on the strategies reported in \cite{Weber2015b}, our own experience as mathematics researchers and teachers, and other sources we propose a {\em Proof understanding baseline checklist} shown in Figure \ref{fig:pubc}.
	
	\begin{figure}
		\begin{enumerate}
			\item Which formal definitions are relevant to the proof?
			\begin{enumerate}
				\item What specialist notation is used, and what does it mean?
				\item Write out definitions which occur in the hypotheses, conclusion and proof, adopting the current notation.
				\item What examples do you know which do/do not satisfy the definitions?
			\end{enumerate}
			\item Describe the overall modular recursive structure of the proof.  E.g.~(i) direct, (ii) definition-chasing, (iii) if and only if, (iv) exhaustive cases, (v) induction, (vi) contrapositive, (vii) contradiction.
			\begin{enumerate}
				\item Identify each structural part of the proof separately. (e.g. an ``if any only if'' proof must have both directions, induction must have (i) a clear hypothesis statement, (ii) base case, (iii) induction step, (iv) conclusion).
				\item Recursively apply the proof understanding baseline checklist to each separate sub-part.
			\end{enumerate}
			\item Hypotheses
			\begin{enumerate}
				\item Where is each hypothesis used in the proof?
				\item In a general proof, which examples do/do not satisfy the hypotheses?  If there is more than one hypothesis, do we have examples which satisfy each logical combination?
			\end{enumerate}
			\item Is a correct warrant justifying each step in the proof given?  If not then provide one.
			\item Does the proof make use of previously known theorems or results? If so, what are they and how are they used?
			\item Does the proof make use of proof-gadgets?  If so, what are they and how are they used?
			\item For an if ... then proof, is the converse true or false?  Do we have counter-examples?
			\item In a general proof, can you follow the proof steps through with a simple specific example, including any proof-gadgets?
		\end{enumerate}
		\caption{Proof understanding baseline checklist}
		\label{fig:pubc}
	\end{figure}
	
	Unlike a pilot's pre-fight checklist, the items listed in the proof understanding baseline checklist are not intended to be used precisely in the order they are written above.
	Indeed, we expect that typically a proof will need to be read a number of times, \cite{Inglis2012}.
	Conversely, in some proofs, the items might be rather trivial and so will not need detailed comment.
	This is an example of where the expert-reversal effect might come into play.
	An experienced mathematician will bring examples to mind, will instinctively search for warrants and will be confident using formal definitions.
	Novices will need prompting, and the proof understanding baseline checklist might be able to help here.
	One hallmark of expertise is knowing the appropriate level of detail, e.g. being able to highlight the distinctive aspects of a particular proof compared to a typical poof in a particular field of mathematics.
	Providing comprehensive examples at various stages (hypotheses, counter example to the converse) might only be appropriate once the overall structure has been appreciated.
	So, the checklist is not intended to be used in a linear fashion by students.
	
	The proof understanding baseline checklist is about a particular proof, and so the checklist consciously neglects some things identified by \cite{Mejia-Ramos2012}, such as transferring the general ideas to another context, or comparing different proofs.
	
	We conjecture that an efficient way to remember many proofs appears to be to identify the overall modular recursive structure of the proof, and how to make use of any gadgets.
	A hallmark of competence would be the ability to reconstruct a substantially complete proof from a statement containing just (i) the overall modular recursive structure of the proof, and (ii) definition and brief statement of how to use any gadgets.
	Indeed, this observation might provide an opportunity for faded worked examples of proofs at a more structural level.
	That is, developing tasks in which students are provided with a proposed structure and the details of a gadget.
	Their task would be to complete a proof.
	
	We anticipate that the proof understanding baseline checklist could be used by students independently. However for many proofs, items within the checklist will be rather trivial, or irrelevant.
	When a teacher creates a set of proof-comprehension questions we expect there will be some craft and artistry in choosing and sequencing the questions which highlight distinctive and interesting features unique to a particular proof.
	
	\section{Proof Comprehension Exercises}\label{sec:dpce}
	
	\cite{Mejia-Ramos2012} developed a proof-comprehension framework, and applied this in \cite{Mejia-Ramos2017} to develop and validate three multiple choice reading comprehension tests.
	They argue that their model {\em ``could also be helpful for professors who teach advanced mathematics courses.''}
	Their process involved the following six stages, see \cite[Table 1]{Mejia-Ramos2017}.
	\begin{enumerate}
		\item Generate open-ended items, covering all aspects of the {\em proof-comprehension framework.}
		\item Conduct pre-test interviews with 12 students, asking them to answer each item generated in stage 1. to generate the choices for multiple-choice tests.
		\item Expert review, and re-draft, of items generated after stage 2.
		\item Conduct a pilot with a further 12 students, asking them to solve problems and think aloud.
		\item Administer the test to a large population (approximately 200 students) to verify that these tests had high internal reliability, to identify problematic items with poor discriminatory power, and to identify items that can be removed to generate shorter multiple-choice tests.
		\item Conduct validating interviews with 12 students, asking them to answer each item generated in stage 5. to verify that the final, shorter, multiple-choice tests accurately measured students' understanding.
	\end{enumerate}
	This process represents a ``gold standard'' for test development, but the effort and resources needed is probably only available in a relatively limited range of situations such as in research projects, or for high-stakes national examinations.
	A similar process is used to develop and validate concept inventories, see \cite{Carlson2010} and \cite{Lane-getaz2013}.
	
	Typical university teachers do not have the resources available to research projects and normally need to generate items on a week-by-week basis.
	So, our original goal was to develop a ``practical taxonomy of question types'' for generating questions which support learning mathematical proof, including reading comprehension.
	We want to equip a thoughtful, practical, teacher with concrete ways to develop assessments of proof, for use as online assessments, of any proof they might want to teach as a routine part of their teaching.
	So, our starting point was the following, abbreviated, process.
	\begin{enumerate}
		\item A member of staff takes the proof to be taught and develops a reading comprehension test.
		\item The test is taken by students, as part of weekly teaching processes.
		\item The teacher evaluates the test, qualitatively and perhaps using some statistics such as item response theory, refines the questions and weeds out some of the poorly performing items.
	\end{enumerate}
	Note that stage 3 of our process it typically not formally undertaken even with relatively large university groups.
	Our process is ambitious and can’t possibly result in the kind of quality instruments developed for research by \cite{Mejia-Ramos2017}.
	But, we have good reason to think it could result in a substantial improvement on current practice.
	The goal of this section is to detail a streamlined, practical process which university teachers can follow in order to quickly create proof comprehension questions for any proof they wish to teach.

	Since we expect students to use the proof understanding baseline checklist shown in Figure \ref{fig:pubc}, an obvious starting point is to create questions directly testing aspects of the checklist.  For example, related to the logical status of statements and proof framework we can always consider asking the following.
	\begin{enumerate}
		\item What is the type of proof?
		\item Identify the lines in the proof in which each hypotheses is used.
		\item Provide a justification for a particular statement/step.
		\item Identify the lines which play particular roles in the proof structure, e.g.~where is the induction hypothesis used in the induction step? E.g. where is the conclusion of the ``if then'' direction, and hypothesis of the ``only if''?
	\end{enumerate}
	Similarly, we expect to ask about definitions, and examples of objects which may satisfy the definitions.
	\begin{enumerate}
		\item Can students identify the correct definition?
		\item Can students choose examples (from a list, say) which satisfy one of the relevant definitions?
		Can we find examples for which each combination of hypotheses holds or fails to hold?  If there are $n$ properties (hypotheses) then there will be $2^n$ examples.
		\item Can students illustrate a statement in a particular case?
		We asked our students to write out \(P(3)\) in the induction proof shown in Figure \ref{fig:induction}.
		Of the 350 attempts online, 26\% of students wrote only \(\frac{3(3+1)(2\times 3+1)}{6}\), confusing the {\em equation} \(P(3)\) with the value of the sum of the series.
		Illustrating specific statements with an example helps to discover any confusion about the meaning of notation, terms and statements in the proof.
	\end{enumerate}
	
	\section{Developing a proof comprehension question}
	
	To develop our thinking of practical proof assessment, we set ourselves the goal of writing at least one proof comprehension question sequence each week in a ``proofs and problem solving'' course.
    ``Proofs and problem solving'' is a year 1 course is taken by  approximately 400 undergraduate students, most of whom are taking mathematics degrees.
	
	In this section we illustrate the ideas discussed by recording our development and use of an example proof comprehension exercise for Theorem \ref{th:bounded-inc-converge}.
	The following is the proof of Theorem \ref{th:bounded-inc-converge} given in \cite[p.~213]{Liebeck2000}, the course textbook for our year 1 course.
	
	\begin{proof}
		Since \((a_n)\) is bounded, the set \(S = \{a_n \: | \: n \in \mathbb{N}\}\) has an upper bound.
		Hence by the Completeness Axiom for \(\mathbb{R}\) (see Chapter 22),  \(S\) has a least upper bound, which we call \(l\). We prove that \(l\) is the limit of \((a_n)_{n=1}^\infty\).
		Let \(\varepsilon \geq 0\).
		Then \(l - \varepsilon\) is not an upper bound for the set \(S\), so there exists an \(N\) such that \(a_N \geq l - \varepsilon\).
		As \((a_n)\) is increasing this implies that \(a_n \geq a_N \geq l - \varepsilon\) for all \(n \geq N\).
		Also, \(l\) is an upper bound for \(S\), so \(a_n \leq l\) for all \(n\).
		We conclude that \(l - \varepsilon \leq a_n \leq l\) for all \(n \geq N\), which means that \(|a_n - l| \leq \varepsilon\) for all \(n \geq N\).
		This shows that \(a_n \to l\).
	\end{proof}
	
	Based on our experiences thus far, we suggest the following factors should be considered when turning such a proof into a reading comprehension exercise.
	
	\subsection{Selecting a Proof}\label{subsec:ps}
	
	There are many reasons why a theorem might be chosen for a proof comprehension question.
	Mathematicians do on occasion talk about the relative importance of particular proofs relative to the discipline as a whole, e.g. \cite{Aigner2004}.
	\cite{Mejia-Ramos2017} suggested that their proofs were chosen to maximize the utility of the corresponding proof comprehension tests.
	More specifically we considered the following factors.
	We require a theorem that is suitable for being taught in our course and is something we would like students to understand.
	Is the proof particularly important in some absolute cultural sense?
	For example, speaking personally, we would like our students to remember the proofs of the following theorems (from our course) long after they graduate, (i) proof of infinitely many primes, (ii) Cantor's diagonal argument and (iii) the proof that \(\sqrt{2}\) is irrational.
	In recalling the proof we would like students to fully understand the details of these proofs.
	It is interesting to note that \cite{Mejia-Ramos2017} selected the first two of these theorems as material for their proof comprehension tests.
	
	Is the proof in a style which is applicable in many situations?
	For example, our goal is that students should learn induction, contradiction and so on.
	Perhaps the proof makes use of important definitions, which are applicable in many situations?
	For example, \(\epsilon/\delta\) arguments appear to require concerted effort for most people to learn.
	
    When writing materials for online assessment we are also driven, in part, by practicality.
	For example, we have a preference for relatively short proofs which will comfortably fit onto a single screen for most users.
	The types of comprehension questions being asked should also be considered when making such a choice.
	Not every proof will naturally yield questions that assess all aspects of proof comprehension, indeed if they did we would have far too many questions.
	For example, specific theorems appear to have fewer opportunities to ask about related examples.
	
	Since proof can be written in different ways, see \cite{Ording2019} for examples, the style of proof should also be considered.
	Different styles of writing can affect the suitability of the proof for various types of proof comprehension.
	
	Our choice of Theorem \ref{th:bounded-inc-converge} was made by following these principles.
	It is a proof that students meet in our course.
	The proof is short, and the result is useful in a practical sense.
	The proof involves important definitions, and the structure of this proof is used elsewhere in analysis.
	Theorem \ref{th:bounded-inc-converge} is a general theorem, involving a number of definitions, with scope for interesting use of examples.
	
	\subsection{Simplifying and Examining the proof}\label{subsec:psimp}
	
	Since our goal is to write proof comprehension questions that can be marked using an online assessment system, it is useful to separate out the steps of the proof and to number them.
	This has the effect of simplifying the layout of the proof whilst providing a way that specific steps of the proof can be referred to and marked online.
    For example, we can re-write the proof of Theorem \ref{th:bounded-inc-converge} as follows.
	
	\begin{proof}$~$\\
		1. Since the sequence \((a_n)\) is bounded, the set \(S = \{a_n \: | \: n \in \mathbb{N}\}\) has an upper bound.
		
		2. Note that \(S\) has a least upper bound, which we call \(l\).
		
		3. We prove that \(l\) is the limit of \((a_n)_{n=1}^\infty\).
		
		4. Let \(\varepsilon \geq 0\) then \(l - \varepsilon\) is not an upper bound for the set \(S\), so there exists an \(N\) such that \(a_N \geq l - \varepsilon\).
		
		5. As \((a_n)\) is increasing this implies that \(a_n \geq a_N \geq l - \varepsilon\) for all \(n \geq N\).
		
		6. Also, \(l\) is an upper bound for \(S\), so \(a_n \leq l\) for all \(n\).
		
		7. Thus, \(l - \varepsilon \leq a_n \leq l\) for all \( n \geq N\), which means that \(|a_n - l| \leq \varepsilon\) for all \(n \geq N\).
		
		8. Therefore this shows that \(a_n \to l\).
	\end{proof}
	
	Numbered lines allows us to ask students about specific statements in the proof without ambiguity.
	Writing statements separately has the added advantage that it somewhat compartmentalises the arguments in the proof and makes it easier to pick out specific items that make up the entire argument.
	If we we wish to test students' understanding of a proof we often mean that we wish to test whether they have understood a more specific idea within the proof.
	For example in our choice of proof, do we wish to know if a student remembers the technique of transferring the argument from sequences to sets as the primary way to apply the completeness hypothesis?
	Is this an exercise in using the formal definition of convergence in a general setting, rather than proving a particular sequence converges?
	
	Breaking down the given argument line by line helps us identify the specific items in the proof.
	It also helps to identify the areas of the proof which are likely to be suitable for online proof comprehension questions.
	Our process was to go through the proof understanding baseline checklist shown in Figure \ref{fig:pubc} and pick out specific components from the checklist which applied to the proof of Theorem \ref{th:bounded-inc-converge}.
	We then authored a comprehension question relating to each of the relevant components. 
	
    For example we may be interested in whether a student recognises the notation for the $\epsilon$-arguments and if they understand the difference between a result holding ``$\forall \: n$" and ``for $n \geq N$".
	In addition we may wish to test if a student knows the correct definition for the least upper bound $l$ or if a student can justify why line 2 in the proof holds.
	We can proceed in this way to build up a selection of potential proof comprehension questions that test the students' knowledge on different aspects of the proof.
	
	\subsection{Relevant Definitions and Modular Structure of the Proof}\label{subsec:lc}
	
	It seems natural to us to start by considering questions relating to the local aspects of the proof.
	Namely the meaning of terms and statements, the logical status of statements and proof framework, and the justification of claims.
	Perhaps the easiest questions to write are those which test the notation or definitions at play in the proof.
	In the proof of Theorem \ref{th:bounded-inc-converge} we have an increasing sequence, a bounded sequence and the least upper bound $b$.
	We also need to use the definition of convergence of a sequence to a limit.
	Each of these definitions readily leads to multiple choice questions which can be marked online.
	
	We can proceed to considering questions relating to the logical status of statements in the proof.
	Again, this provides a wide scope for proof comprehension questions.
	For Theorem \ref{th:bounded-inc-converge} we could ask the following example questions.
	\begin{enumerate}
		\item[] In which statements do we use the assumption that $(a_n)$ is increasing?
		\item[] In which statements do we use the assumption that $(a_n)$ is bounded?
	\end{enumerate}
	
	However, a teacher may not only be interested in testing students' understanding of the statements in the proof, but also interested in testing whether students can provide warrants in moving from one statement to the next.
		These warrants could be contained within the proof or could be external.
	For our proof we could ask the following questions.
	\begin{enumerate}
		\item[] Why does $S$ have an upper bound in statement 2?
		\item[] Why can we proceed from statement 4. to statement 5. in the proof?
	\end{enumerate}
	Once again, these questions may be turned into suitable MCQ's.
	Notice that the second sentence of the original proof starts ``Hence by the Completeness Axiom...'', whereas we have deliberately dropped this in our proof to provide an opportunity to ask students for the warrant for this step.
	
	We could also ask questions about steps in a proof which rely on a previous theorem or result.
	For example in Theorem \ref{th:bounded-inc-converge} we could ask what result we must appeal to in statement 2 in order to see that the set $S$ has a least upper bound?
	Note that questions of this sort are subtly different to those discussing warrants between steps.
	In this case we are asking the student to provide some backing for the step, following Toulmin's terminology \cite{Toulmin2003}.
	Whether questions of this sort are suitable is ultimately a practical teaching decision depending on the context in which the student is seeing the question.
	
	This proof makes use of the proof-gadget \(S\).
	The only way to apply the completeness axiom directly is via a set with an upper bound.
	In constructing \(S\) we turn a sequence into a set of values.
	The resulting upper bound turns out to be the limit of the sequence.
	We could ask students about the motivation for constructing the set \(S\) in this way.
	
	\subsection{Generating Examples and Counterexamples}\label{subsec:ep}
	
	Another potentially fruitful area for proof comprehension is to create questions associated with some of the global aspects of the proof.
	Asking students to illustrate concepts within the proof by providing examples is one such way to test understanding and limitations of the proof.
	
	Theorem \ref{th:bounded-inc-converge} and its proof make use of four definitions.
	We potentially have eight situations to consider all possible combinations of the three definitions in the hypotheses and conclusion.
	The initial goal is to either provide an example which satisfies the combination or to explain why such an example does not exist.
	This is shown in Table \ref{th:bounded-inc-converge-examples}.
	The first example, \(a_n=1-\frac{1}{n}\), is increasing and bounded, and so converges.
	This exemplifies the theorem.
	An example which is increasing and bounded, but does not converge would be a counter example to the theorem.
	Given our proof, such an example cannot occur but do all students understand this?
	If a sequence converges then it is bounded, so two rows in the table correspond to potential counter examples to this, separate theorem (note A).
	Hence these examples cannot exist.
	Do students understand the difference in status between a hypothesis which is necessary for the truth of the theorem (bounded), and a hypothesis which is used to enable the proof-gadget to ``work''?
    That is to say, the increasing assumption is used to show the supremum of the proof-gadget \(S\) is indeed the limit of the sequence.
	The remaining combinations, i.e.~other rows in the table, can be exemplified.
	
	\begin{table}[h!]
		\centering
		\begin{tabular}{ccc|l}
			{\bf Increasing ?} & {\bf Bounded ?} & {\bf Converges ?} & {\bf Example} \\
			\hline
			T         &      T    &    T       & Exemplify theorem: \(a_n=1-\frac{1}{n}\) \\
			T         &      T    &    F       & Counter example to theorem! \\
			T         &      F    &    T       & Note A. \\
			T         &      F    &    F       & \(a_n=n\) \\
			F         &      T    &    T       & \(a_n=1/n\) \\
			F         &      T    &    F       & \(a_n=(-1)^n\) \\
			F         &      F    &    T       & Note A. \\
			F         &      F    &    F       & \(a_n=(-n)^n\) \\
		\end{tabular}
		\caption{Example combinations for Theorem \ref{th:bounded-inc-converge}}
		\label{th:bounded-inc-converge-examples}
	\end{table}
	
	In a general theorem, examining the table of example combinations provides a potentially rich source of questions for students.
	Students can be asked to
	\begin{enumerate}
		\item identify which properties a particular example satisfies;
		\item provide examples satisfying certain properties;
		\item justify why certain examples do/do not exist.
	\end{enumerate}
	For general theorems this appears to be something which could be done in a potentially systematic way.
	
	\subsection{Application of Proof Techniques to a Different Context}
	
	Linked to the idea of illustrating by examples is the notion of taking techniques used in the proof and applying them to other situations.
	We note that this is somewhat different than applying the proof to other situations.
	For example in our proof we may wish to ask a student if they can use the $\epsilon$-arguments when applied to a different proof rather than asking for an example of an increasing, bounded sequence.
	Similarly, we could ask the following questions.
	
	\begin{enumerate}
		\item[] Would the proof still work if we considered a bounded, decreasing sequence? If not, why not?
		\item[] Could we write an alternative proof in order to prove that a bounded increasing sequence converges and, if so what are the relative merits of each proof?
		\item[] What is the statement of the contrapositive of the theorem? Are there other equivalent formulations?
	\end{enumerate}
	
	These ideas relate to the ability of students to understand the precise scope of the proof under consideration.
	Indeed, many proofs are introduced in undergraduate courses because the theorem which is proved is required for the use of several examples.
	This process often involves the application of the conclusion of the proof, rather than any technique used within the proof.
	In addition to applying the result of the theorem, we may also be interested to see if students can adapt the methods or structure of the proof itself in a practical way.
	For example, we may wish to assess if students can run through a general proof on a specific example or if students understand what happens to any proof-gadgets when the proof is applied to a specific example.
	We could also assess students' understanding of ways in which a given proof could break. That is, if we apply the proof to an example where the theorem is false, then where does the proof break?  E.g. in the classic proof that $\sqrt{2}$ is irrational, what goes wrong {\em in the proof} when we try to apply the proof to show $\sqrt{4}$ is irrational?
	In this manner, we can create potential comprehension questions around the global concerns of the proof.
	A simple proof comprehension asking five basic questions about the proof of Theorem \ref{th:bounded-inc-converge} is shown in Figure \ref{fig:Increase-bound-converge}.

	\begin{figure}
		\begin{center}
			\includegraphics[width=12cm]{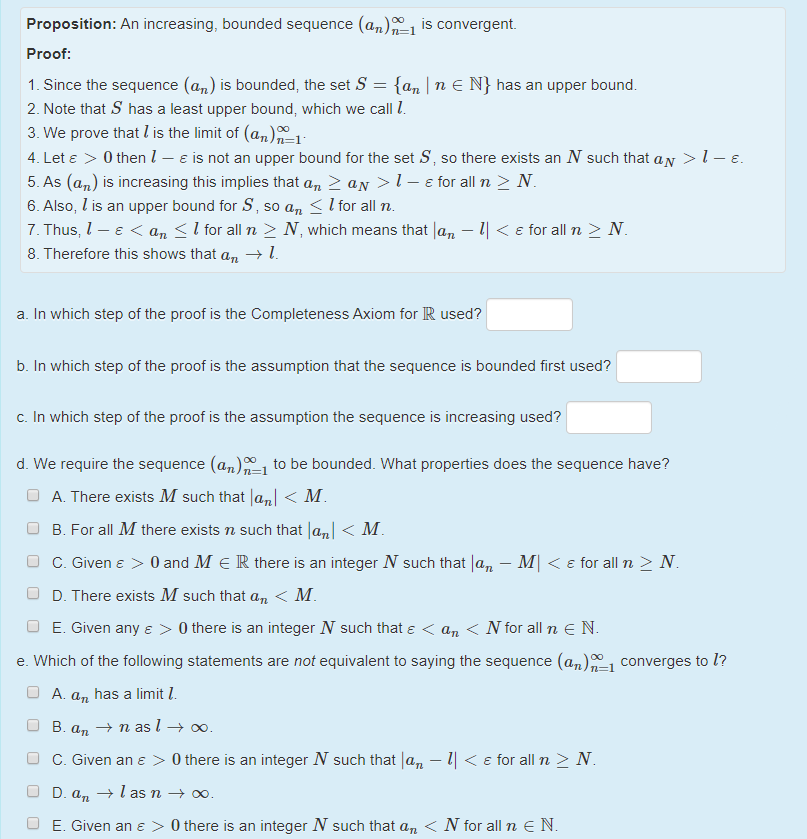}
			\caption{A proof comprehension question}\label{fig:Increase-bound-converge}
		\end{center}
	\end{figure}

	\subsection{Results}
	
	We asked the question shown in Figure \ref{fig:Increase-bound-converge} during semester 2 of the 2019-20 session, to a year 1 group of undergraduate students.
	The question was included in an online pre-lecture test as part of a flipped-classroom cycle.
	Students had one attempt at the question, and feedback was provided after the test had closed.
	We had 344 responses to the question, although not all students answered every part of the question.
	
	(a) In which step of the proof is the Completeness Axiom for \(\mathbb{R}\) used?
	This question was answered correctly (line 2) by 243 (70.64\%) and incorrectly by 95 (27.62\%).
	Common incorrect responses were line 1: 56 (16.57\%), or line 4: 23 (6.80\%), or line 5: 6 (1.78\%).
	Line 1 refers to the upper-bound of \(S\), as does the completeness axiom and so is understandable as a choice.
	However, we only {\em use} the completeness axiom when we apply the conclusion, i.e.~the existence of the least upper bound.
	Similarly, lines 4 and 6 also refer to upper-bounds suggesting some students are focusing on surface terminology.
	
	(b) In which step of the proof is the assumption that the sequence is bounded first used?
	This question was answered correctly (line 1) by 261 (75.87\%) and incorrectly by 78 (22.67\%).
	Complete results are as follows.
	\begin{center}
		\begin{tabular}{ll}
			Line & Responses \\
			1 & 261 (76.99\%)\\
			2 & 19 (5.60\%)\\
			3 & 14 (4.13\%)\\
			4 & 34 (10.03\%)\\
			6 & 11 (3.24\%)\\
		\end{tabular}
	\end{center}
	We do not have a hypothesis for the incorrect responses to this question, since boundedness of the {\em sequence} (rather than the set \(S\)) is not used anywhere else.
	
	(c) In which step of the proof is the assumption the sequence is increasing used?
	This question was answered correctly (line 5) by 303 (89.38\%) and incorrectly by 32  (9.30\%).
	Common incorrect answers were line 4: 18  (5.31\%) or line 3: 8 (2.36\%).
	We do not have a hypothesis for the incorrect responses to this question, since the hypothesis that the sequence is increasing is not used anywhere else.
	
	(d) We require the sequence \((a_n)_{n=1}^\infty\) to be bounded. What properties does the sequence have?
	\begin{itemize}
		\item[A.] There exists \(M\) such that \( |a_n| < M \).
		\item[B.] For all \(M\) there exists \( n \) such that \( |a_n|< M \).
		\item[C.] Given \( \varepsilon > 0\) and \( M \in \mathbb{R} \) there is an integer \(N\) such that \(|a_n -M| < \varepsilon\) for all \(n \geq N\).
		\item[D.] There exists \(M\) such that \(a_n < M\).
		\item[E.] Given any \(\varepsilon > 0\) there is an integer \(N\) such that \(\varepsilon < a_n < N\) for all \(n \in \mathbb{N}\).
	\end{itemize}
	Students were asked to choose all those which applied.  16 students failed to respond at all.
	Only 85 (24.71\%) of students chose the correct two responses (A and D).
	In addition to the 85 (24.71\%) correct selections, students chose the following.
	\begin{center}
		\begin{tabular}{llll}
			Choice  & Correct? & Selected & Not selected.\\
			A       & Y        & 133 (38.66\%) & 109 (31.69\%)\\
			B       &          & 43  (12.50\%) & 199 (57.85\%)\\
			C       &          & 197 (57.27\%) & 45  (13.08\%)\\
			D       & Y        & 147 (42.73\%) & 95  (27.62\%)\\
			E       &          & 63  (18.31\%) & 179 ( 52.03\%)
		\end{tabular}
	\end{center}
	Hence, only 63.37\% of students correctly selected A, and only 67.44\% of students correctly selected D.
	C looks plausible, but potentially confuses bounded with convergence (although C is deliberately not the correct definition of convergence).  A majority of students chose C.
	
	(e) Which of the following statements are {\em not} equivalent to saying the sequence \((a_n)_{n=1}^\infty\) converges to \(l\)?
	\begin{itemize}
		\item[A.] \(a_n\) has a limit \(l\).
		\item[B.] \(a_n \to n\) as \(l \to \infty\).
		\item[C.] Given an \(\varepsilon > 0\) there is an integer \(N\) such that \(|a_n -l| < \varepsilon\) for all \(n \geq N\).
		\item[D.] \(a_n \to l\) as \(n \to \infty\).
		\item[E.] Given an \(\varepsilon > 0\) there is an integer \(N\) such that \(a_n < N\) for all \(n \in \mathbb{N}\).
	\end{itemize}
	Students were asked to choose all those which applied.  14 students failed to respond at all.
	In this case, 208 (60.47\%) of students chose the correct two responses (B and E).
	In addition to the 208 (60.47\%) correct selections, students chose the following.
	\begin{center}
		\begin{tabular}{llll}
			Choice  & Correct? & Selected & Not selected.\\
			A       &          & 52  (15.12\%) & 70  (20.35\%)\\
			B       & Y        & 75  (21.80\%) & 47  (13.66\%)\\
			C       &          & 61  (17.73\%) & 61  (17.73\%)\\
			D       &          & 57  (16.57\%) & 65  (18.90\%)\\
			E       & Y        & 35  (10.17\%) & 87  (25.29\%)
		\end{tabular}
	\end{center}
	Hence, 82.27\% of students correctly selected B, and only 70.64\% of students correctly selected E.
	
	We expect pre-lecture test questions to have high success rates, and in many cases our instincts appear to have been validated by the response rates of students.
	These questions were easy, but still, not all students got the correct answers.
	In the case of the question shown in Figure \ref{fig:Increase-bound-converge} one criticism from a colleague we feel is somewhat legitimate is that a ``test-wise'' student could use their knowledge of ``the game'' to guess answers.
	For example, the correct answer to part (b) (line 1) is merely the first time the word ``bounded'' is used in the proof.
	All assessment formats (MCQ, open-ended, reading comprehension) have their own strengths/weaknesses, and some formats have unintended format effects.
	These results are an honest report of the question in Figure \ref{fig:Increase-bound-converge} which is just one of our relatively early attempts at developing a sequence of proof comprehension questions.
	No doubt it can be improved.
	In the next section we discuss general issues raised by our attempts at developing proof comprehension questions.
	
	\section{Discussion}

	We wish to understand to what extent our proof comprehension exercises have been worthwhile: to what extent have they led to an greater understanding of proof in our students?
	An examination of the students' most common responses to each question allows us to identify the most common mistakes, and to identify any floor/ceiling effects with impossible/trivial questions.
	A more thorough long-term evaluation has been disrupted by COVID-19 during March 2020.
	Given the sudden increase in the use of online learning and assessments as a response to this global health crisis, and our cautious belief of the modest contribution practical proof comprehension has to play in this, we have chosen to publish these preliminary findings how, rather than wait a full year for experimental data from a more controlled setting.
	
	The STACK online assessment system we are using allows a question setter to provide immediate tailored feedback to a student, based on predefined rules.
	However, the question writer must predict common student errors and provide suitable feedback for each of these.
	We have been highly pragmatic in not trying to second-guess too many possible incorrect responses, but rather to examine actual student responses.
	For example, for part (a) in the question in Figure \ref{fig:Increase-bound-converge}, $16.57\%$ of students incorrectly believed the completeness axiom was used in line 1.
	Therefore we latter added the feedback ``Both line 1 and the completeness axiom refer to the upper bound for $S$ however we only use the completeness axiom in line 2 to find the least upper bound.''
	This gives the students who answered line 1, specific feedback correcting their error after the test, and this feedback will be in place for subsequent years.
	We can repeat this procedure with all common errors for all questions after the materials have been used.
	We accept there is a delay in developing this feedback for students who may not (this year) go back to review their responses.
	
	\begin{figure}[h!]
		\begin{center}
			\includegraphics[width=7cm]{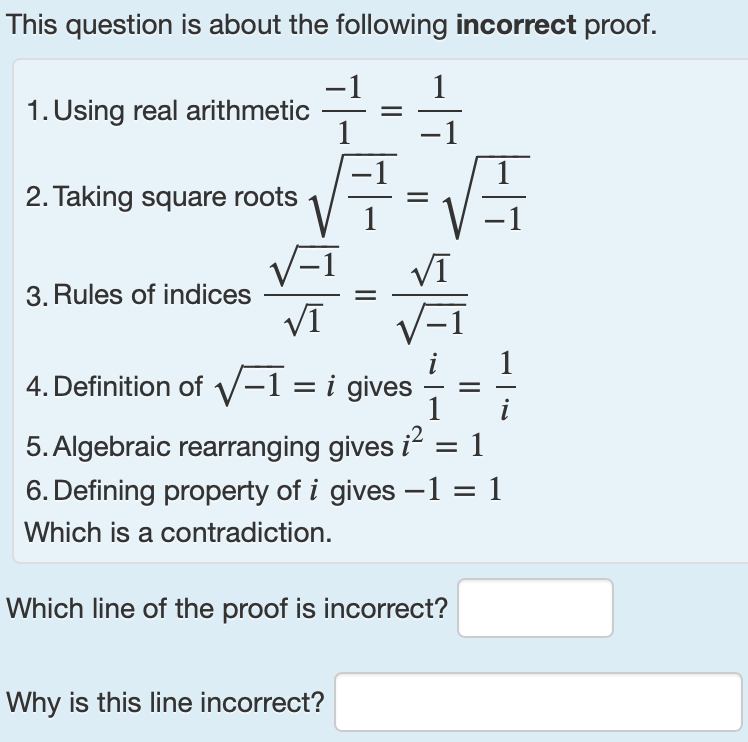}
			\caption{A proof fallacy question}\label{fig:pcw5a}
		\end{center}
	\end{figure}
	
	In some questions we asked students to give their reasoning as free-form text. Free-form text answers are not automatically assessed and these responses were not assigned any marks in the quiz.
    For example in the question shown in Figure \ref{fig:pcw5a} students were asked for the statement in the proof that contained an error.
	This was the answer that was automatically assessed, however the students were also asked to explain their reasoning and this explanation was not marked automatically.
	Often, students would enter notes that we did not anticipate.
	For example, when completing the question in Figure \ref{fig:pcw5a} one response was that the proof fails because ``You cannot take the square root of a negative number." 
	We altered our feedback to this question as a result for students who chose line 2.
	In many cases students had the correct answer but gave incorrect reasoning or vice versa. 
	For this question the most common response was the correct answer (line 3), selected by 47.68\% of students. 
	However other common responses were line 4 (22.60\%) and line 2 (16.41\%), suggesting students were having difficulty in correctly identifying the error in the proof.
	The main difficulty encountered by the students here appeared to be the consequences of defining $i^2 = 1$ rather than defining $i = \sqrt{-1}i$.
	Taking the comments in addition to the common question responses allowed for a more accurate identification of the precise areas of the proof where students appear to be struggling.
	This helped us to develop more accurate feedback, addressing the common errors for this question. 
	By applying this process to each question we can provide tailored feedback which can be supplemented each time the question materials are used.
	Based on our experiences of gathering free-form answers we are confident that we could gather students' responses in year 1, and from these develop multiple choice questions for subsequent student cohorts.
	Students in year 1 would not get any benefit from feedback from the automatically scored MCQ, and either we would have to accept that or have a human assess students' justifications provided in year 1.

    To what extent does the proof style inform the types of comprehension question that can be developed from it?
	In \cite{Ording2019} various styles of proof writing are discussed in a general way.
	This is of interest here because, for the purposes of reading comprehension, it is possible to alter the style of a given proof in order to increase the viability of developing comprehension questions.
	For example, the induction proof in Figure \ref{fig:induction2} uses subtle colour and style which is intended to highlight structural components of the induction proof, providing cues to the student.
	This style has the potential to highlight the modular structure of a proof which may be otherwise obscured. 

	When re-writing the proof of Theorem \ref{th:bounded-inc-converge} using numbered steps, we purposely omitted the explicit reference to the Completeness Axiom in statement 2, between the proof provided by \cite{Liebeck2000} and our proof with numbered lines.
	By consciously choosing to omit a warrant we can then ask students to provide this information as an answer to a comprehension question.
	See Figure \ref{fig:Increase-bound-converge}, statement 2 and question (a).
	In a similar way, it is possible to purposely omit particular algebraic expressions in a proof for the purpose of asking students to provide that information by filling in gaps.
	We note the subtle difference between deliberately obfuscating logical steps of a proof and obfuscating calculations within a proof.
	Another possible style choice would be to provide the proof to the student in a two column format and to obscure some of the working or the justification steps we wish students to provide.

	The complex fallacy question shown in Figure \ref{fig:pcw5a} could ask students to provide a warrant for each line, and to identify any lines which are incorrect.
	This has the advantage that we do not deliberately make the proof more difficult to understand.
	In addition, this style allows for the precise identification of the warrant we wish students to provide.
	While the self-explanation training of \cite{Hodds2014} has proved to be effective, explicitly asking students specific questions about a proof might (i) encourage them to ask such questions about any proof they encounter, and (ii) where carefully designed might prompt students to consider aspects of the proof they might otherwise not consider.

	We wish to insert a small note of caution at this point.
	One unexpected difficulty encountered during the development of our proof comprehension questions was that we found in certain cases, subtleties in either the structure or reasoning of a proof which were not necessarily straightforward for us to follow.
	This was despite the fact many of these proofs are well known arguments.
	We hypothesise that this was a result of our very familiarity with these proofs.
	Over time we have become trained to subconsciously internalise some of their structure as we have gained expertise.
	This can have the effect that we under-appreciate some of the subtleties of an argument.
	Appropriate proof comprehension development therefore requires the ``rediscovery" of every level of the proof.
	
	What has perhaps surprised us the most in developing and trialing our questions is the lack of floor/ceiling effects, even with apparently trivial questions.
	Despite teaching high-achieving students on a mathematics degree at a leading university, many of our students are not able to identify lines in a proof correctly.
	The inability of many students to correctly identify lines of a proof strongly suggests the value of using such questions to gauge students' understanding.

	\section{Conclusion}

	Developing proof comprehension tasks has allowed us to test students' understanding of specific aspects of important proofs. 
	In addition, it has helped us to more accurately identify the specific areas of each proof where students appear to be struggling. 
	Once we had developed, and gained confidence in using, the proof understanding baseline checklist shown in Figure \ref{fig:pubc} we have been able to readily and efficiently produce adequate proof comprehension question sequences on a week-by-week basis. By testing each appropriate item of the checklist, we were able to create at least one suitable proof comprehension question sequence per week. 
	The self-imposed requirement that all of the questions we developed were suitable for marking online, and the lack of resources to develop high-quality MCQ was a restriction on what we could do.
	This has been significantly simpler than the gold-standard process described in \cite{Mejia-Ramos2017}.

	Each week we have been able to review the results of each question. Overall there was a lack of floor/ceiling effects in most of our questions. 
	Based on preliminary examination of students' attempts we believe these questions have been acceptable, or better, tests of important aspects of students' understanding of these proofs.
	
	At times developing proof comprehension questions has been a very enjoyable process, which has deepened our understanding of the proofs.
	At other times it has proved rather frustrating, with a growing confusion over some proofs as currently written in the core text.
	We suspect a more comprehensive examination of standard proofs through the attempt to apply the proof understanding baseline checklist can only lead to improvements (indeed perhaps correction of errors/omissions) in the proofs we provide to students.
	However, we have been able to write a variety of comprehension questions, including reading comprehension, fading questions and generating examples questions.
	Using each of these techniques, we have been able to test a range of different aspects of proof understanding.
	We hope that the methods described here can aid university teachers to produce comprehension tasks across a wide variety of subjects within mathematics, and that such questions are a genuine aid to developing students' understanding of proof itself.
	
	\bibliography{education,PUS,sangwin}
	
	\end{document}